\documentclass[a4paper]{amsart}
\usepackage{amssymb}
\usepackage[T1]{fontenc}
\usepackage[colorlinks, linkcolor=blue]{hyperref}
\usepackage{calrsfs}
\theoremstyle{plain}
\newtheorem{main}{Theorem}

\newtheorem{theorem}{Theorem}[section]
\newtheorem{lemma}[theorem]{Lemma}
\newtheorem{corollary}[theorem]{Corollary}
\newtheorem{proposition}[theorem]{Proposition}

\theoremstyle{definition}
\newtheorem{definition}[theorem]{Definition}
\newtheorem{example}[theorem]{Example}

\theoremstyle{remark}
\newtheorem{remark}[theorem]{Remark}

\DeclareMathOperator{\ord}{ord}
\DeclareMathOperator{\mdim}{mdim}

\DeclareMathOperator{\umdim}{\overline{mdim}_M}
\DeclareMathOperator{\lmdim}{\underline{mdim}_M}
\DeclareMathOperator{\omdim}{mdim^o_M}
\DeclareMathOperator{\capa}{cap}
\DeclareMathOperator{\ocap}{ocap}
\DeclareMathOperator{\supp}{supp}
\DeclareMathOperator{\spn}{span}

\begin{document}

\title[Mean dimension of general IFS]{Mean dimension of general iterated function systems}

\author[W. Cordeiro]{Welington Cordeiro}
\address{Instituto de Matemática e Computação, Universidade Federal de Itajubá
	Avenida BPS 1303, 37500-903 Pinheirinho, Itajubá MG
	Brazil}
\email{welington@unifei.edu.br}
\thanks{W. Cordeiro was partially supported by CAPES-Finance Code 001 and by CNPq Grant number 173057/2023-3}
\author[M. J. Pacifico]{Maria José Pacifico}
\address{Instituto de Matemática, Universidade Federal do Rio de Janeiro - UFRJ, Cidade Universitária - Ilha do Fundão, Rio de Janeiro  21945-909. Brazil}
\email{pacifico@im.ufrj.br}
\thanks{M. J. Pacifico was partially supported by CAPES-Finance Code 001, CNPq Projeto Universal No. 404943/2023-3, CNPq-Brazil grant 307776/2019-0 and by
	Foundation for Research Support of the State of Rio de Janeiro (FAPERJ) grant CNE
	E-26/200.913/2022(268181).}
\author[X. Zhang]{Xuan Zhang}
\address{Instituto de Matemática e Estatística, Universidade Federal Fluminense, Niterói, Rio de Janeiro, Brazil}
\email{xuanz@id.uff.br}
\thanks{X. Zhang was partially supported by CAPES-Finance Code 001}

\date{\today}


\begin{abstract}
In this paper, we introduce and investigate the notions of Mean Dimension and Metric Mean Dimension for generalized iterated function systems (IFS). We establish basic properties of these invariants and prove that Mean Dimension is always bounded above by the lower Metric Mean Dimension and the upper Metric Mean Dimension in this setting. We further show that generalized iterated function systems with the Small Boundary Property have zero Mean Dimension. Finally, we introduce a Gluing Orbit Property for generalized iterated function systems and prove that, under suitable transitivity and non-rigidity assumptions, it guarantees positive topological entropy.
\end{abstract}

\maketitle
\tableofcontents

\section{Introduction}

In \cite{G}, Gromov introduced the notion of \emph{mean dimension}, a topological invariant designed to capture the complexity of dynamical systems whose phase spaces may be infinite-dimensional. In contrast to topological entropy, one of the most classical quantitative measures of dynamical complexity, mean dimension is particularly effective in situations where entropy alone is too coarse to detect the size of the underlying dynamics. Motivated by this perspective, Lindenstrauss and Weiss \cite{LindenstraussWeiss2000} introduced the notion of \emph{metric mean dimension}, establishing a fundamental connection between topological and metric aspects of complexity. In particular, they proved that finite topological entropy implies zero metric mean dimension.

Since then, mean dimension and metric mean dimension have proved to be robust and flexible tools, and several variants have been developed in different dynamical settings. In \cite{RA}, Rodrigues and Acevedo extended metric mean dimension to non-autonomous dynamical systems. Later, Cheng and Li \cite{CL} introduced the notion of upper metric mean dimension for impulsive semiflows, a class of systems exhibiting discontinuous jumps in time. More recently, Tang, Ye, and Ma \cite{TYM} defined metric mean dimension for free semigroup actions on non-compact spaces. These developments highlight both the versatility of mean dimension theory and its relevance for increasingly general forms of dynamical behavior.

On the other hand, in \cite{DenkerYuri2015}, Denker and Yuri introduced a generalized class of iterated function systems (IFS), allowing the domains of the maps to vary. This framework substantially enlarges the classical notion of IFS and provides a natural setting in which symbolic, geometric, and thermodynamic aspects of dynamics interact. Generalized iterated function systems naturally generate families of orbit structures that are richer than those arising in a single-map dynamical system. Because of this, they provide a compelling framework to test whether dimensional invariants from topological dynamics remain meaningful beyond the classical autonomous setting.

For this generalized class, Denker, Cordeiro, and Yuri \cite{DenkerYuri2015, CordeiroDenkerYuri2015, Denker2023} developed a rich thermodynamic formalism, including notions of entropy, pressure, and conformal families of measures. Among other results, they showed that IFS with the specification property have positive entropy \cite{CordeiroDenkerYuri2015}, established the existence of conformal families of measures under finite pressure \cite{DenkerYuri2015}, and obtained estimates for Hausdorff dimension in the expanding case via a generalized Bowen--Manning--McCluskey type formula \cite{Denker2023}. Despite this substantial progress, the dimensional complexity of generalized IFS from the viewpoint of mean dimension has, to the best of our knowledge, not yet been investigated.

The purpose of this paper is to initiate such a study. More precisely, we introduce the notions of \emph{mean dimension} (denoted by $\mdim$) and \emph{upper and lower metric mean dimensions} (denoted by $\umdim$ and $\lmdim$) for generalized IFS in the sense of Denker and Yuri \cite{DenkerYuri2015}, and we establish foundational results showing that the basic structure of the classical theory extends to this broader setting.

Our first main result shows that these dimensional invariants satisfy the expected hierarchy.

\begin{main}\label{T1}
	Let $(X,\mathfrak{V})$ be an iterated function system. Then
	$$
	\mdim(X,\mathfrak{V}) \leq \lmdim(X,\mathfrak{V}) \leq \umdim(X,\mathfrak{V}).
	$$
\end{main}

Theorem \ref{T1} places mean dimension and metric mean dimension in a coherent structural framework for generalized IFS. Its proof requires the introduction of an intermediate invariant, which we call the \emph{orbit metric mean dimension}. This notion is of independent interest, as it provides a natural bridge between the topological and metric viewpoints and clarifies the role played by orbit complexity in this generalized setting.

We also extend to generalized IFS the notion of the \emph{Small Boundary Property} (SBP), introduced by Lindenstrauss and Weiss in \cite{LindenstraussWeiss2000}. In the classical theory, the SBP is closely related to the vanishing of mean dimension, and our second main result shows that this phenomenon persists in the present framework.

\begin{main}\label{T2}
	Let $(X,\mathfrak{V})$ be an iterated function system with the SBP. Then
	$$
	\mdim(X,\mathfrak{V}) = 0.
	$$
\end{main}

Theorem \ref{T2} provides a concrete dynamical criterion ensuring the vanishing of mean dimension for generalized IFS, thereby extending one of the central structural implications of the classical theory. Taken together, Theorems \ref{T1} and \ref{T2} show that generalized iterated function systems admit a meaningful and robust mean dimension theory, and they provide a first step toward a broader dimensional approach to this class of systems.

Another aspect of the dynamical complexity of generalized iterated function systems that we address in this work concerns the relation between orbit-gluing properties and topological entropy. Motivated by analogous phenomena in classical dynamical systems, we introduce the notion of the \emph{Gluing Orbit Property} for generalized iterated function systems and investigate its dynamical consequences. Under mild recurrence and transitivity assumptions, we show that the absence of uniformly rigid behavior guarantees positive topological entropy. More precisely,

\begin{main}\label{T3}
	Let $(X,{\mathfrak V})$ be an iterated function system with the Gluing Orbit property, with $Tran(X,{\mathfrak V},\sigma)\neq\emptyset$ for some $\sigma\in \Sigma$ that is not uniformly rigid. 
	Then $h(X,{\mathfrak V})>0$. 
\end{main}

Theorem \ref{T3} shows that the Gluing Orbit Property imposes a substantial degree of dynamical complexity on generalized iterated function systems. This provides a useful mechanism for detecting chaotic behavior in this non-autonomous and highly flexible setting.
Roughly speaking, this property allows one to concatenate finitely many prescribed orbit segments and shadow them with a single orbit, up to uniformly bounded transition times. In this setting, we prove that the coexistence of orbit-gluing and a suitable transitivity assumption forces the system to exhibit positive topological entropy, provided that a certain rigidity obstruction is absent.

The paper is organized as follows. In Section~\ref{sec:mdim}, we recall the definition of generalized iterated function systems and introduce the notions of mean dimension and metric mean dimension, together with some of their basic properties. In Section~\ref{sec:t1}, we prove Theorem~\ref{T1}. In Section~\ref{sec:sbp}, we define the Small Boundary Property for generalized IFS and prove Theorem~\ref{T2}.  In Section~\ref{sec:gop} we introduce the Gluing Orbit Property and relevant notions in Theorem~\ref{T3}, whose proof is given in Section~\ref{sec:t3}.

\section{Mean Dimension}\label{sec:mdim}
Let $X$ be a compact metric space and let $\mathfrak{V}$ be a family of homeomorphisms $v:D(v)\rightarrow v(D(v))\subset X$  with closed domain $D(v)\subset X$. The pair $(X,\mathfrak{V})$ is called a function system. It is called an {\em iterated function system} if there is a set $O\subset X$ such that,
\[\emptyset\neq \bigcup_{v\in\mathfrak{V}}v(D(v)\cap O)\subset\bigcup_{v\in\mathfrak{V}}D(v)\cap O.\]
This notion has been studied in \cite{DenkerYuri2015} where some general motivation to investigate iterated function systems of this general type is also discussed.

\begin{example}
	First we remember the definition of the metric of Gromov-Hausdorff on the space of all compact metric spaces. Let $X$ and $Y$ be compact metric spaces. The triple $(\tilde{X},\tilde{Y},Z)$ consisting of a compact metric space $Z$ and its subsets $\tilde{X}$ and $\tilde{Y}$, isometric to $X$ and $Y$, respectively, is called a \textit{realization} of $(X,Y)$. We define the \textit{Gromov-Hausdorff distance} between $X$ and $Y$ by
	$$d_{GH}(X,Y):=\inf\{r;\text{ there is a realization }(\tilde{X},\tilde{Y},Z)\text{ with }d_H(\tilde{X},\tilde{Y})\leq r\},$$ 
	where $d_H(A,B)$ is the \textit{Hausdorff} distance between $A$ and $B$, two compact metric subsets of a same metric space. By \cite[ Corollary 2.7]{IvandovIlliadisTuzhilin2016}, for each pair of compact metric spaces $X$ and $Y$, there is  a realization $(\tilde{X},\tilde{Y},Z)$ such that  
	$$d_{GH}(X,Y)=d_H(\tilde X, \tilde Y).$$
	Such realization will be called an \textit{optimal realization} for $X$ and $Y$. More generaly, by induction we say that $(\tilde{X}_1,\tilde{X}_2,..., \tilde{X}_n, \tilde{X}_{n+1}, Z_{n+1})$ is an \textit{optimal realization} of an ordered finite set $\{X_1, X_2,..., X_n, X_{n+1}\}$ of compact metric spaces, as an optimal realization for $ Z_n$ and $X_{n+1}$. Note that, in general, optimal realizations are not unique, even in the case of only two compact metric spaces.
	  
	Now consider a finite family $F=\{f_i:X_i\rightarrow X_i; 1\leq i \leq n\}$ of homeomorphisms $f_i$'s defined on compact metric spaces $X_i$'s. 
We fix an optimal realization $(\tilde{X}_1,\tilde{X}_2,..., \tilde{X}_n, Z_n)$.   Then $(Z_n, F)$ is a function system.
This example show that we can create a generalized function system using any finite set of dynamical systems defined in any finite number of compact metric spaces.
\end{example}

Let $(X,\mathfrak{V})$ be an iterated function system. We begin with a modification of topological entropy for the case of an IFS (see \cite{CordeiroDenkerYuri2015, Walters1982}). The basic idea is to count the minimal number of orbits from $(X,\mathfrak{V})$ which are needed to shadow all  orbits.  This leads to the following definitions.

Let $\Sigma=\{(v_1,v_2,...,v_n,...);v_i\in\mathfrak{V}\}$. According to  the definition of IFS, there exists $\sigma=(v_1,v_2,...)\in\Sigma$ and $x\in X$ such that 
$v_n\circ v_{n-1}\circ...\circ v_1(x)$ is well defined for each $n\in\mathbb N=\{1, 2,\ldots\}$. 
For a pair $(x, \sigma)\in X\times \Sigma$, let $v^{\sigma(n)}(x)=v_n\circ v_{n-1}\circ...\circ v_1(x)$ whenever it is well defined and $v^{\sigma(0)}(x)=x$, and define the \emph{orbit} of $(x,\sigma)$ as
\[O(x,\sigma)=\{x,v^{\sigma(1)}(x),v^{\sigma(2)}(x),...,v^{\sigma(n)}(x),...\}.\]
We also let 
 $v^{\sigma(n,m)}= v_n\circ v_{n-1}\circ...\circ v_{m+1}$ for $n>m\in\mathbb Z^+=\{0, 1, 2,\ldots\}$, so $v^{\sigma(n)}=v^{\sigma(n,0)}$.
Define $$\Sigma_x=\{\sigma\in\Sigma; v^{\sigma(n)}(x) \ \text{is well defined for all}\  n\in\mathbb{N} \},$$
and
$$\Sigma_\sigma=\{x\in X; v^{\sigma(n)}(x) \ \text{is well defined for all}\  n\in\mathbb{N} \}.$$
Both $\Sigma_x$ and $\Sigma_\sigma$ can be empty.

Let  \[\Psi(X)=\{ (x,\sigma): x\in X, \sigma\in \Sigma_x\}.\] We call a set 
$R\subset \Psi(X)$ an \emph{$(n,\epsilon)$-spanning set} of $X$, if for all $(x,\sigma)\in \Psi(X)$  there is $(y,\varphi) \in R$,  such that
\[d(v^{\sigma(i)}(x),v^{\varphi(i)}(y))<\epsilon, \ \forall i\in\{0, 1,...,n-1\}. \]
We denote the minimum cardinality of an $(n,\epsilon)$-spanning set for $X$ by $r(n,\epsilon)$.  Since we are not restricting the notion to the case of a finite family $\mathfrak {V}$, $r(n,\epsilon)$ may be infinite.
We call a set $S\subset \Psi(X)$ an \emph{$(n,\epsilon)$-separated set} for $X$, if for all pairs $(x,\sigma), (y,\varphi)\in S$ there is some   $i\in\{0, 1,...,n-1\}$ satisfying
\begin{equation*} d(v^{\sigma(i)}(x),v^{\varphi(i)}(y))\geq \epsilon.
\end{equation*}
We denote the maximum cardinality of an $(n,\epsilon)$-separated set for $X$ by  $s(n,\epsilon)$.

The \emph{topological entropy} of a IFS $(X,\mathfrak{V})$ is then given by
\begin{eqnarray*}
	h(X,\mathfrak{V})&:=&\lim_{\epsilon\rightarrow 0}\limsup_{n\rightarrow\infty}\frac{1}{n}\log  r(n,\epsilon) \\
	&=& \lim_{\epsilon\rightarrow 0}\limsup_{n\rightarrow\infty}\frac{1}{n}\log  s(n,\epsilon).
\end{eqnarray*}
The existence of the limits and their equality were proved in \cite{CordeiroDenkerYuri2015}.
Define the \emph{upper metric mean dimension of $(X,\mathfrak{V})$}
\begin{eqnarray*}
	\umdim(X,\mathfrak{V})&:=&\limsup_{\epsilon\rightarrow 0}\frac{1}{|\log(\epsilon)|}\limsup_{n\rightarrow\infty}\frac{1}{n}\log  r(n,\epsilon)
\end{eqnarray*}
and the \emph{lower metric mean dimension of $(X,\mathfrak{V})$}
\begin{eqnarray*}
	\lmdim(X,\mathfrak{V})&:=&\liminf_{\epsilon\rightarrow 0}\frac{1}{|\log(\epsilon)|}\limsup_{n\rightarrow\infty}\frac{1}{n}\log  r(n,\epsilon).
\end{eqnarray*}
When they are equal, define the \emph{metric mean dimension} to be the common value. The metric mean dimension depends on the metric $d$ on $X$. 
\begin{remark}
As in \cite{LindenstraussWeiss2000}, it is clear that the metric mean dimension of a generalized IFS is zero if it has finite topological entropy.
\end{remark}

Let $A\subset X$ be a closed subset. If $\alpha$ is a finite open cover of $A$, we say that a cover $\beta$ of $A$ refines $\alpha$, if for each $U\in\beta$ there is $V\in\alpha$ such that $U\subset V$. In this case we write $\beta\succ\alpha$. If $\alpha$ is an open cover of $A\subset X$ we define,
\begin{eqnarray*}
	\ord(\alpha,A):=\max_{x\in A}\sum_{U\in\alpha}1_U(x)-1  \text{ and } \mathcal{D}(\alpha):=\min_{\beta\succ\alpha}\ord(\beta,A),
\end{eqnarray*} 
where $\beta$ runs over finite open covers of $A$ refining $\alpha$. It is clear that $\mathcal{D}(\alpha)$ is subadditive, i.e. $$\mathcal{D}(\alpha\vee\beta)\leq\mathcal{D}(\alpha)\vee\mathcal{D}(\beta).$$

Let $\alpha$ be an open cover of $X$.
For a continuous map $f:X\rightarrow Y$, we say that $f$ is \emph{$\alpha$-compatible} if there is an open cover $\beta$ of $f(X)\subset Y$ such that $f^{-1}(\beta)\succ\alpha$. 
 We say that $(X,\mathfrak{V})$ is \emph{$\alpha$-compatible} if for each $v\in\mathfrak{V}$ there is an open cover $\beta$ of $v(D(v))$ such that $v^{-1}(\beta)\succ\alpha$. Define $$\mathfrak{V}(\alpha)=\{v\in\mathfrak{V};v\text{ is $\alpha$-compatible}\}.$$
By definitions, $(X,\mathfrak{V})$ is $\alpha$-compatible if and only if $\mathfrak{V}(\alpha)=\mathfrak{V}$. 
The next result is  \cite[Proposition 2.3]{LindenstraussWeiss2000}.
 
\begin{lemma}\label{lem:comp} 
If for each $y\in Y$, $f^{-1}(y)\subset U\in\alpha$ for some $U\in\alpha$, then $f$ is $\alpha$-compatible.  
\end{lemma}
Suppose $\alpha$ is an open cover of $X$ and $\sigma=\{v_1,v_2,...,v_n,...\}\in\Sigma$, we denote
$$v^{-\sigma(n)}\alpha=\{v_1^{-1}\circ ... \circ v_{n-1}^{-1}\circ v_n^{-1}(U); U\in\alpha\}$$ 
$v^{-\sigma(0)}\alpha=\alpha$ and for $b>a\in\mathbb{Z^+}$ 
$$\alpha_a^b(\sigma)=v^{-\sigma(a)}\alpha\vee v^{-\sigma(a+1)}\alpha\vee ... \vee v^{-\sigma(b)}\alpha.$$

\begin{definition}
	If $(X,\mathfrak{V})$ is an iterated function system, define for each $\sigma\in\Sigma$ 
	$$\mdim(X,\mathfrak{V},\sigma)=\sup_{\alpha\in\mathcal{U}_X}\lim_{n\to+\infty}\frac{\mathcal{D}(\alpha_0^{n-1}(\sigma))}{n},$$
	where $\mathcal{U}_X$ is the set of all finite open covers of $X$, 
	and define the \emph{mean dimension} of $(X,\mathfrak{V})$ as $$\mdim(X,\mathfrak{V})=\sup_{\sigma\in\Sigma}\mdim(X,\mathfrak{V},\sigma).$$
\end{definition}
For each $\sigma\in\Sigma$, the limit in definition of $\mdim(X,\mathfrak{V},\sigma)$ exists by the appendix of \cite{LindenstraussWeiss2000}. As in the classical definition of mean dimension for homeomorphisms, we have that for each $\sigma\in\Sigma$ 
$$\mathcal{D}(\alpha_0^{n-1}(\sigma))\leq \sup_{\alpha\in\mathcal{U}_X}\mathcal{D}(\alpha)=:\dim X.$$ 
Therefore, if the topological dimension of $X$ is finite, then $$\mdim(X,\mathfrak{V})=0.$$

From the definition it is easy to see the following lemma.
\begin{lemma}
	Let $(X,\mathfrak{V})$ be an iterated function system and $\sigma\in\Sigma$. If $\{\alpha_k\}_{k=1}^{+\infty}$ is a sequence of open covers of $X$ such that the maximal diameter of an element of $\alpha_k$ is $\frac{1}{k}$, then
	$$\mdim(X,\mathfrak{V},\sigma)=\sup_{k\in\mathbb{N}}\lim_{n\rightarrow+\infty}\frac{\mathcal{D}(\alpha_0^{n-1}(\sigma,k))}{n},$$  
where $\alpha_0^{n-1}(\sigma,k)=(\alpha_k)_0^{n-1}(\sigma)$.
\end{lemma}

\begin{proof} It is clear that 
$$\mdim(X,\mathfrak{V},\sigma)\geq\sup_{k\in\mathbb{N}}\lim_{n\rightarrow+\infty}\frac{\mathcal{D}(\alpha_0^{n-1}(\sigma,k))}{n}.$$  
On the other hand, notice that if $\alpha\in\mathcal{U}_X$ then there is $k\in\mathbb{N}$ such that 
$$\alpha_k\succ\alpha.$$
Therefore,
$$\mathcal{D}(\alpha_0^{n-1}(\sigma,k))\geq \mathcal{D}(\alpha_0^{n-1}(\sigma)).$$ 
Hence, 
$$\mdim(X,\mathfrak{V},\sigma)\leq\sup_{k\in\mathbb{N}}\lim_{n\rightarrow+\infty}\frac{\mathcal{D}(\alpha_0^{n-1}(\sigma,k))}{n}.$$
\end{proof}

For $v\in\mathfrak V$ define $\sigma_v=(v,v,v,...,v,...)$. The following result is a version of  \cite[Proposition 2.7]{LindenstraussWeiss2000}.

\begin{proposition}\label{p1}
	Let $(X,\mathfrak{V})$ be an iterated function system and 
	$v\in\mathfrak V$. If for some $n\in\mathbb{N}$ we have that 
	$v^n\in\mathfrak V$, then
	$$\mdim(X,\mathfrak{V},\sigma_{v^n})=n\cdot \mdim(X,\mathfrak{V},\sigma_v).$$
\end{proposition}

\begin{proof} If $\alpha$ is an open cover of $X$, we have that
\begin{eqnarray*}
	&&\lim_{k\rightarrow +\infty}\frac{\mathcal{D}(v^{-\sigma_{v^n}(0)}\alpha\vee v^{-\sigma_{v^n}(1)}\alpha\vee ... \vee v^{-\sigma_{v^n}(k-1)}\alpha)}{k}\\
	&=& \lim_{k\rightarrow +\infty}\frac{\mathcal{D}(\alpha\vee v^{-n}\alpha\vee ... \vee v^{-(k-1)n}\alpha)}{k}\\&\leq& 
	n\cdot \lim_{k\rightarrow\infty}\frac{\mathcal{D}(\alpha_0^{kn-1}(\sigma_v))}{kn}\leq n\cdot \mdim(X,\mathfrak{V},\sigma_v).  
\end{eqnarray*}
Therefore,
$$\mdim(X,\mathfrak{V},\sigma_{v^n})\leq n\cdot \mdim(X,\mathfrak{V},\sigma_v).$$
Since,
$$\alpha_0^{kn-1}(\sigma_v)= \alpha_0^{n-1}\vee v^{-n}\alpha_0^{n-1}\vee ...\vee v^{-(k-1)n}\alpha_0^{n-1},$$
we have that
$$\mdim(X,\mathfrak{V},\sigma_{v})=\sup_{\alpha\in\mathcal{U}_X}\lim_{k\rightarrow\infty}\frac{\mathcal{D}(\alpha_0^{kn-1}(\sigma_v))}{kn}\leq \frac{\mdim(X,\mathfrak{V},\sigma_{v^n})}{n}.$$
Hence,
$$\mdim(X,\mathfrak{V},\sigma_{v^n})=n\cdot \mdim(X,\mathfrak{V},\sigma_v).$$
\end{proof}

\begin{corollary}\label{c1}
	Let $(X,\mathfrak{V})$ be an iterated function system and $v^n\in\mathfrak{V}$ for each $n\in\mathbb{N}$, such that $\mdim(X,\mathfrak{V},\sigma_v)>0$. Then $\mdim(X,\mathfrak{V})=+\infty$.
\end{corollary}

\begin{proof}
By Proposition \ref{p1}, $$\mdim(X,\mathfrak{V},\sigma_{v^n})=n\cdot \mdim(X,\mathfrak{V},\sigma_v)\rightarrow+\infty, \text{ } n\rightarrow+\infty.$$
And since $$\mdim(X,\mathfrak{V})=\sup_{\sigma\in\Sigma}\mdim(X,\mathfrak{V},\sigma),$$ we have that $\mdim(X,\mathfrak{V})=+\infty$.
\end{proof}

\begin{example}
\begin{enumerate}
	\item If $s$ is the shift on $([0,1]^d)^\mathbb{Z}$, by Proposition 3.3 of \cite{LindenstraussWeiss2000}, $\mdim(([0,1]^d)^\mathbb{Z},s)=d$. Defining $\mathfrak{V}=\{s^n;n\in\mathbb{N}\}$, by Corollary \ref{c1}, 
$$\mdim(([0,1]^d)^\mathbb{Z},\mathfrak{V})=+\infty.$$
	\item Let $X=D(v)=\mathbb{T}^2\cup ([0,1]^d)^\mathbb{Z}$ and $v$ be the Linear Anosov $A$ on $\mathbb{T}^2$ and the shift $s$ on $([0,1]^d)^\mathbb{Z}$. Since $\mdim(\mathbb{T}^2,A)=0$ and $\mdim(([0,1]^d)^\mathbb{Z},s)=d$, we have that $\mdim(X,v)=d$. If $D(w)=\mathbb{T}^2$, $w$ is the identity on $\mathbb{T}^2$ and $v,w\in\mathfrak{V}$, then  $$\mdim(X,\mathfrak{V},\sigma_v)=d \ \text{ and } \mdim(X,\mathfrak{V},\sigma)=0,$$
	where $\sigma=(w,v,v,...,v,...)$.   
\end{enumerate}
\end{example}

\section{Proof of Theorem \ref{T1}}\label{sec:t1}
We are now in a position to prove Theorem \ref{T1}, namely, that for every iterated function system $(X,\mathfrak{V})$,
\[
\mdim(X,\mathfrak{V})\leq \lmdim(X,\mathfrak{V})\leq \umdim(X,\mathfrak{V}).
\]

Before proceeding, let us emphasize the meaning of this statement. Since metric mean dimension depends on the choice of the metric $d$ on $X$, the theorem shows that the mean dimension is always bounded above by the lower metric mean dimension, for \emph{every} metric on $X$ compatible with the topology. In this sense, Theorem \ref{T1} establishes mean dimension as a genuinely topological lower bound for metric mean dimension in the setting of generalized iterated function systems.

The proof relies on the introduction of an intermediate quantity, which we call the \emph{orbit metric mean dimension}. This notion plays a key role in connecting the topological and metric viewpoints, and allows us to place it naturally between mean dimension and metric mean dimension.

We call a set $S\subset \Psi(X)$ an \emph{$(\sigma,n,\epsilon)$-separated} set, if for all pairs $(x,\sigma), (y,\sigma)\in S$ there is some   $i\in\{0, 1,...,n-1\}$ satisfying
\begin{equation*} d(v^{\sigma(i)}(x),v^{\sigma(i)}(y))\geq \epsilon.
\end{equation*}
We denote the maximum cardinality of an $(\sigma,n,\epsilon)$-separated set by $s(\sigma,n,\epsilon)$.
Then define the \emph{metric mean dimension of $(X,\mathfrak{V})$ in $\sigma$} by
\begin{eqnarray*}
	\omdim(X,\mathfrak{V},\sigma)&:=&
	\liminf_{\epsilon\rightarrow 0}\frac{1}{|\log(\epsilon)|}\limsup_{n\rightarrow\infty}\frac{1}{n}\log  s(\sigma,n,\epsilon)
\end{eqnarray*}
and the \emph{orbit metric mean dimension of $(X,\mathfrak{V})$} by
\begin{eqnarray*}
	\omdim(X,\mathfrak{V})&:=&\sup_{\sigma\in\Sigma}\omdim(X,\mathfrak{V},\sigma).
\end{eqnarray*}

\begin{lemma}\label{l2}
	Let $(X,\mathfrak{V})$ be an iterated function system. Then 
	$$\omdim(X,\mathfrak{V})\leq \lmdim(X,\mathfrak{V}).$$
\end{lemma}

\begin{proof}
Fix $\sigma\in\Sigma$. Note that each   $(\sigma,n,\epsilon)$-separated set also is a $(n,\epsilon)$-separated set. Therefore,
$$\omdim(\sigma,X,\mathfrak{V})\leq \lmdim(X,\mathfrak{V}).$$
Hence,
$$\omdim(X,\mathfrak{V})=\sup_{\sigma\in\Sigma} \omdim(\sigma,X,\mathfrak{V})\leq \lmdim(X,\mathfrak{V}).$$
\end{proof}

\begin{lemma}\label{131}
	Let $(X,\mathfrak{V})$ be an iterated function system and $\sigma\in\Sigma$. Then 
	$$\mdim(X,\mathfrak{V},\sigma)\leq \omdim(X,\mathfrak{V},\sigma).$$
\end{lemma}
\begin{proof}
By definition of $\mdim(X,\mathfrak{V},\sigma)$ it is enough to show that $$\lim_{n\to\infty}\frac{\mathcal{D}(\alpha_0^{n-1}(\sigma))}{n}\leq \omdim(X,\mathfrak{V},\sigma),$$
for every $\alpha\in\mathcal{U}_X$.

As in \cite{LindenstraussWeiss2000}, we can 
refine $\alpha$ to the form of
$$\alpha=\{U_1,V_1\}\vee\{U_2,V_2\}\vee...\vee\{U_r,V_r\},$$
where each $\{U_i, V_i\}$ is an open cover of $X$.
Define $w_i:\Sigma_\sigma\rightarrow [0,1]$ by
$$w_i(x)=\frac{d(x,X-V_i)}{d(x,X-V_i)+d(x,X-U_i)}. $$
Then $w_i$ is Lipschitz, $U_i=w_i^{-1}[0,1)$ and $V_i=w_i^{-1}(0,1]$.
Let $$C=\max_{1\le i\le r}\{|L_i|\in\mathbb{R}; L_i \text{ is  Lipschitz constant of } w_i\}.$$
And for each natural number $N$, define $$F_\sigma(N,\cdot):\Sigma_\sigma\rightarrow [0,1]^{rN}$$ by
\begin{multline*}
F_\sigma(N,x)=(w_1(x),...,w_r(x),w_1(v^{\sigma(1)}x),...,w_r(v^{\sigma(1)}x),...,\\w_1(v^{\sigma(N)}x),...,w_r(v^{\sigma(N)}x)).\end{multline*}
Then 
$F_\sigma(N,\cdot)$ is $ \alpha_0^{N-1}(\sigma)$-compatible. 
The proofs of the next two lemmas follows the same lines of  \cite[Lemmas 4.3 and 4.4]{LindenstraussWeiss2000}.

For any $S\subset\{1,...,rN\}$ and $z\in [0,1]^{rN}$  let $z_S\in [0,1]^{|S|}$ be the projection to the coordinates in $S$. 
For $\sigma\in\Sigma$, $n\in\mathbb{N}$ and $x,y\in\Sigma_\sigma$ define
$$d^\sigma_n(x,y)=\max\{d(x,y),d(v^{\sigma(1)}(x),v^{\sigma(1)}(y)),...,d(v^{\sigma(n)}(x),v^{\sigma(n)}(y))\}.$$

\begin{lemma}\label{4.3}
	Let $\epsilon>0$ be given. Then there is a natural number $N_\epsilon$ such that if $N>N_\epsilon$, then there is $\xi\in (0,1)^{rN}$ such that if $|S|\geq (D+\epsilon)N$, $$\xi_S\notin F_\sigma(N,\Sigma_\sigma)_S,$$  
	where $D=\omdim(X,\mathfrak{V},\sigma)$.
\end{lemma}

\begin{proof}
Let $\delta>0$ satisfy that $$\delta<(2^r(2C)^{2D})^{-2/\epsilon} \ \text{ and } \ \frac{\limsup_{n\rightarrow 0}\frac{1}{n}\log s(\sigma,n,\delta)}{|\log\delta|}\leq 
D+\frac{\epsilon}{4}.$$
If $N$ is large, we can cover the closure of $\Sigma_\sigma$ by $\delta^{-(D+\frac{\epsilon}{2})N}$ dynamical balls $$B_\sigma(x,N,\delta)=\{y\in\Sigma_\sigma;d_N^\sigma(x,y)<\delta\}.$$ Therefore, if we define $||(a_1,...,a_{rN})-(b_1,...,b_{rN})||
=\sup_i|a_i-b_i|$ then $$F_\sigma(N,B_\sigma(x,N,\delta))\subset\{a\in[0,1]^{rN};||F_\sigma(N,x)-a||<C\delta\}.$$
Hence $F_\sigma(N,\Sigma_\sigma)$ can be covered by $\delta^{-(D+\frac{\epsilon}{2})N}$ balls of radius $C\delta$ in the $||.||$ norm. 
 We can name these balls $B(1),...,B(K)$, where $K=\delta^{-(D+\frac{\epsilon}{2})N}$. 
Choose $\xi\in [0,1]^{rN}$ with uniform probability and we have that
$$\mathbb{P}(\xi_S\in F_\sigma(N,\Sigma_\sigma)_S)\leq \sum_{j=1}^K\mathbb{P}(\xi_S\in B(j)_S)\leq \delta^{-(D+\frac{\epsilon}{2})N}(2C\delta)^{|S|}.$$
Therefore,
\begin{eqnarray*}
	&&\mathbb{P}(\exists S;|S|\ge(D+\epsilon)N
\text{ and } \xi_S\in F_\sigma(N,\Sigma_\sigma)_S)  \\&\leq & \sum_{|S|\ge(D+\epsilon)N} \mathbb{P}(\xi_S\in F_\sigma(N,\Sigma_\sigma)_S) \\ &\leq&\sum_{|S|\ge(D+\epsilon)N} \delta^{-(D+\frac{\epsilon}{2})N}(2C\delta)^{(D+\epsilon)N} \leq 2^{rN}((2C)^{2D}\delta^{\frac{\epsilon}{2}})^N\ll 1.
\end{eqnarray*}
Then, 
with positive probability, a random $\xi$ satisfies the lemma.
 \end{proof}

\begin{lemma}\label{4.4}
	If $\pi:F_\sigma(N,\Sigma_\sigma)\rightarrow [0,1]^{rN}$ satisfies that for all $\xi\in [0,1]^{rN}$
	$$ \{1\leq k\leq rN;\xi_k=a\}\subset\{1\leq k\leq rN;\pi(\xi)|_k=a\},$$ for $a\in \{0,1\}$, then $\pi\circ F_\sigma(N,\cdot)$ is compatible with $\alpha_0^{N-1}(\sigma)$. 
\end{lemma}

\begin{proof}
For each $\xi\in [0,1]^{rN}$, $0\leq j<N$ and $1\leq i<r$ define $$W_{i,j}=\begin{cases}
	v^{-\sigma(j)}(U_i), \ \text{if} \ \xi_{jr+i}=0, \\
	v^{-\sigma(j)}(V_i), \ \text{otherwise}.
\end{cases} $$  		
Hence, $$(\pi \circ F_\sigma(N,\cdot))^{-1}(\xi)\subset \bigcap_{i,j}W_{i,j}\in \alpha_0^{N-1}(\sigma),  $$
by definition. Therefore, $\pi\circ F_\sigma(N,\cdot)$ is $\alpha_0^{N-1}(\sigma)$-compatible by Lemma \ref{lem:comp}.
\end{proof}

Now come back to the demonstration of Lemma \ref{131}.
Fix $\epsilon>0$ and let $\xi^1$ and $N$ be given by Lemma \ref{4.3}. Define
$$\Phi=\{\xi\in[0,1]^{rN};\xi_k=\xi^1_k \text{ for more than } (D+\epsilon)N \text{ indices } k\}, $$ 
then $F_\sigma(N, \Sigma_\sigma)\subset \Phi^c:=[0,1]^{rN}\setminus\Phi$.
And for $m\in\mathbb{N}$, define $$J_m:=\{\xi\in [0,1]^{rN}; \xi_i\in\{0,1\} \text{ for at least } m \text{ indices } 1\leq i\leq rN\}.$$
Now we 
define a function $\pi$ satisfying the hypothesis of Lemma \ref{4.4}. First, we define $$\pi_1:[0,1]^{rN}\setminus\{\xi^1\}\rightarrow J_1,$$ by mapping each $\xi$ to the intersection of the ray originates from $\xi^1$ and moving through $\xi$ and $J_1$ (this is possible since the interior of $[0,1]^{rN}$ contains $\xi^1$). 
For each of the $(rN-1)$-dimensional cubes $I$ that comprise $J_1$,
a retraction on 
$I$ can be defined similarly, with the projection of $\xi^1$ onto 
$I$ serving as the center. This will establish a continuous retraction, denoted by $\pi_2$, from $\Phi^c$ into $J_2$. As long as there is an intersection between $\Phi$ and the cubes in $J_m$, this process can be iterated, resulting in the desired continuous projection $\pi$ of $\Phi^c$ onto $J_{m_0}$, with $$m_0+[(D+\epsilon)N]+1= rN.$$
By Lemma \ref{4.4}, $\pi\circ F_\sigma(N,\cdot)\succ\alpha^{N-1}_0(\sigma).$ And since $F_\sigma(N,\Sigma_\sigma)\subset\Phi^c$, we have that $$\pi\circ F_\sigma(N, \Sigma_\sigma)\subset J_{m_0}. $$  
As the topological dimension of $J_{m_0}$ is $[(D+\epsilon)N]+1$,  then by \cite[Proposition 2.4]{LindenstraussWeiss2000} $\mathcal{D}(\alpha^{N-1}_0(\sigma))\leq (D+\epsilon)N+1.$
Hence 
$$\mdim(X,\mathfrak{V},\sigma)\leq D=\omdim(X,\mathfrak{V},\sigma).$$
\end{proof}

\begin{lemma}\label{l3}
	Let $(X,\mathfrak{V})$ be an iterated function system. Then 
	$$\mdim(X,\mathfrak{V})\leq \omdim(X,\mathfrak{V}).$$
\end{lemma} 
\begin{proof}
Since 
$$\mdim(X,\mathfrak{V})=\sup_{\sigma\in\Sigma}\mdim(X,\mathfrak{V},\sigma) \ \text{ and } \omdim(X,\mathfrak{V})=\omdim(X,\mathfrak{V},\sigma),$$
by Lemma \ref{131} we have the desired inequality.
\end{proof} 

\begin{proof}[Proof of Theorem \ref{T1}]
It is a direct consequence of Lemma \ref{l3} combined with Lemma \ref{l2}. 
\end{proof}

\section{Small-Boundary Property and Proof of Theorem \ref{T2}}\label{sec:sbp}
In this section we will generalize the small-boundary property to generalized IFS. Let $(X,\mathfrak{V})$ be an iterated function system, and $A\subset X$. If $n>0$ and $x\in A$, we define the \emph{n-capacity} of $x$ relative to $A$ as the number:
$$\capa(n,x,A)=\sup_{\sigma\in\Sigma_x}\frac{1}{n}\sum_{i=0}^{n-1}1_A(v^{\sigma(i)}(x)),$$
and the \emph{orbit capacity} of $A$ as :
$$\ocap(A)=\lim_{n\rightarrow+\infty}\sup_{x\in X}\capa(n,x,A).$$
We say that $A$ is \emph{small} if the orbit capacity of $A$ is zero.
An iterated function system $(X,\mathfrak{V})$ has the \emph{small-boundary property} (SBP) if for each point $x\in X$ and every open neighborhood $U$ of $x$, there is an open set $V\subset U$ that has a small boundary and $x\in V$.

Now we are going to extend \cite[Theorem 5.4]{LindenstraussWeiss2000} for iterated functions systems. The next lemma is an essential ingredient of the proof.
\begin{lemma}\label{LSBP} Let $(X,\mathfrak{V})$ be an iterated function system with the SBP property and $\sigma\in\Sigma$. Then for every open cover $\alpha$ of $\Sigma_\sigma$ and every $\epsilon>0$ there is a subordinate partition of the unity $\phi_j: \Sigma_\sigma\rightarrow [0,1]$ ($j=1,...,|\alpha|$) depending on $\sigma$, such that 
	\begin{enumerate}
		\item $\sum_{j=1}^{|\alpha|}\phi_j(x)=1$;
		\item for each $j=1,...,|\alpha|$, there is $U\in\alpha$ such that $\supp(\phi_j)\subset U$;
		\item $\ocap(\bigcup_{j=1}^{|\alpha|}\phi_j^{-1}(0,1))<\epsilon$.  
	\end{enumerate}
\end{lemma}		

\begin{proof}
Since $(X,\mathfrak{V})$ has SBP, there is a refinement $\beta\succ\alpha$ such that there is 
one-to-one correspondence between the elements $U_j$ of $\alpha$ and $V_j$ of $\beta$ with $\overline{V_j}\subset U_j$, and the elements of $\beta$ have small boundaries.

Let $N>0$ such that for each $j\in\{1,...,|\alpha|\}$ and for each $x\in \Sigma_\sigma$,
$$\frac{1}{N}\sum_{i=0}^{N-1}1_{\partial V_j}(v^{\sigma(i)}(x))<\frac{\epsilon}{|\alpha|}.$$
Hence, denoting $$B_{\delta,j}=B_\delta(\partial V_j)=\{y\in \Sigma_\sigma;d(y, \partial V_j)<\delta\},$$ there is $\delta>0$ such that
$$\frac{1}{N}\sum_{i=0}^{N-1}1_{B_{\delta,j}}(v^{\sigma(i)}(x))<\frac{\epsilon}{|\alpha|},$$
and $B_{\delta,j}\subset U_j$, for each $j\in\{1,...,|\alpha|\}$. Define

$$ \psi_j(x)=\begin{cases}
	1, \ \text{if} \ x\in V_j, \\
	\max(0,1-\frac{d(x,\partial V_j)}{\delta}), \ \text{otherwise}.
\end{cases} $$
And finally we define the desired functions by induction:
$$\begin{cases}\phi_1(x)=\psi_1(x) \\
	\phi_{i+1}(x)=\min(\psi_{i+1}(x),1-\sum_{l=1}^{i}\phi_l(x)).
	 \end{cases}$$

\end{proof}

\subsection{Proof of Theorem \ref{T2}}
We now turn to the proof of Theorem \ref{T2}. More precisely, we show that if $(X,\mathfrak{V})$ is an iterated function system satisfying the SBP, then its mean dimension vanishes:
\[
\mdim(X,\mathfrak{V})=0.
\]		   

\begin{proof}
Fix $\sigma\in\Sigma$ and $\alpha$ a cover of $X$. For a given $\epsilon>0$ take a subordinate partition of unity $\{\phi_j\}$ as in Lemma \ref{LSBP}. 
Define $A=\bigcup_{i=1}^{|\alpha|}\phi_j^{-1}(0,1)$. Let $N$ be large enough such that 
\begin{equation}\label{eq1} 
\frac{1}{N}\sum_{i=1}^{N-1}1_A(v^{\sigma(i)}(x))<\epsilon
\end{equation}
for each $x\in\Sigma_\sigma$. Define $\Phi:\Sigma_\sigma\rightarrow [0,1]^{|\alpha|}$ 
by \[\Phi(x)=(\phi_1(x),...,\phi_{|\alpha|}(x)),\]
and map $f_N:\Sigma_\sigma\rightarrow\mathbb{R}^{|\alpha|N}$ by $$f_N(x)=(\Phi(x),\Phi(v^{\sigma(1)}(x)),...,\Phi(v^{\sigma(N-1)}(x))).$$
Let $e_j^i$, $i=1,...,N$, $j=1,...,|\alpha|$ be the canonical basis of $\mathbb{R}^{|\alpha|N}$. For each $M<\epsilon N$, $I=\{i_1,...,i_M\}$ and $\eta\in\{0,1\}^{|\alpha|N}$ define 
$$C(I,\eta)=\spn\{e_j^i;i\in I, \ 1\leq j \leq|\alpha|\}+\eta.$$
By \eqref{eq1}, we have that $$f_N(\Sigma_\sigma)\subset \bigcup_{\eta\in\{0,1\}^{|\alpha|N}}\bigcup_{|I|<\epsilon N}C(I,\eta).$$
Hence, $f_N$ is $\alpha_0^{N-1}(\sigma)$-compatible. By \cite[Proposition 2.4]{LindenstraussWeiss2000},
$$\mathcal{D}(\alpha_0^{N-1}(\sigma))<\epsilon |\alpha|N.$$
Therefore, $$\mdim(X,\mathfrak{V},\sigma)=0.$$
Since $\sigma$ is arbitrary, 
$$\mdim(X, \mathfrak V)=0.$$
\end{proof}

\section{Gluing Orbit Property}\label{sec:gop}
As before, $(X,{\mathfrak V})$ denotes an IFS. For a natural number $M$, consider the space of sequences $\Sigma_M:=\{1,2,...,M\}^\mathbb{N}$ with entries in $\{1,2,...,M\}$. 
\begin{definition}  
	An \textit{orbit sequence} is a finite or infinite sequence$$\mathcal{C}=\{(x_j,\sigma_j,m_j)\}_{1\leq j\leq N},$$ with $N$ finite or infinite, 
	of ordered triples such that $x_j\in X$, $\sigma_j\in\Sigma_{x_j}$ and $m_j\in\mathbb{N}$. If $N$ is finite, then $m_N$ is allowed to be $+\infty$. 
	A \textit{gap} 
	is a sequence of positive integers $$\mathcal{G}=\{t_j\}_{1\le j\le N-1}.$$ For $\epsilon>0$ we say that $(\mathcal{C},\mathcal{G})$ is \textit{$\epsilon$-traced} by $(z,\sigma)\in \Psi(X)$ 
	if for $1\le j\le N$ 
	$$d(v^{\sigma(s_j+l)}(z),v^{\sigma_j(l)}(x_j))<\epsilon \text{ for } l=0, \ldots, m_j-1,$$
	where $$s_1:=0 \text{ and } s_j:=\sum_{i=1}^{j-1}(m_i+t_i-1) \text{ for } 2\le j\le N. $$
	We say that $(X,{\mathfrak V})$ has the \textit{Gluing Orbit Property} if for each $\epsilon>0$ there is $M=M(\epsilon)>0$ such that for each orbit sequence $\mathcal{C}$, there is a gap $\mathcal{G}\in\Sigma_M$ such that $(\mathcal{C},\mathcal{G})$ can be $\epsilon$-traced.
\end{definition}

\begin{definition}	
	\begin{itemize}
		\item $A\subset \mathbb{N}$ is \emph{syndetic} if there is $L>0$ such that $$A\cap[n,n+L-1]\neq\emptyset, \text{ for every }n\in\mathbb Z^+.$$
		\item 
		For $(x,\sigma)\in\Psi(X)$ 
		and $\epsilon>0$, denote $$R(x,\sigma,\epsilon):=\{n\in\mathbb{Z}^+; d(v^{\sigma(n)}(x),x)<\epsilon\},$$
		and $$R(\epsilon):=\bigcap_{x\in X}\bigcap_{\sigma\in\Sigma_x}R(x,\sigma,\epsilon).$$
		\item $x\in X$ is \emph{almost periodic} if $R(x,\sigma,\epsilon)$ is syndetic for every $\sigma\in\Sigma_x$ and $\epsilon>0$.
		\item $(X,{\mathfrak V})$ is \emph{uniformly almost periodic} if $R(\epsilon)$ is syndetic for every $\epsilon>0$.
		\item We say that 
		$\sigma\in\Sigma$ is \emph{equicontinuous} if for every $\epsilon>0$ there is $\delta>0$ such that for any $x,y\in\Sigma_\sigma$ with $d(x,y)<\delta$, then $$d(v^{\sigma(n)}(x),v^{\sigma(n)}(y))<\epsilon, \text{ for each } n\in\mathbb Z^+.$$
	\end{itemize}

\end{definition}

For $\sigma=(v_1, v_2, \ldots)$ and $n\in\mathbb{Z}^+$ let $\sigma(+\infty, n)=(v_{n+1}, v_{n+2}, \ldots)$.

\begin{definition}
	We say that 
	$\sigma\in\Sigma$ is \emph{uniformly continuous} if for each $k\in\mathbb{Z}^+$ 
	and every $\epsilon>0$ there is $\delta=\delta(\sigma,k,\epsilon)>0$ such that for each $n\in\mathbb Z^+$
	and $1\leq j\leq k$, if $x,y\in\Sigma_{\sigma(+\infty,n)}$ with $d(x,y)<\delta$ then $$d(v^{\sigma(n+j,n)}(x),v^{\sigma(n+j,n)}(y))<\epsilon.$$
\end{definition}
\begin{remark}
If $\sigma=(v_i)_{i\in\mathbb{N}}$ is such that the set $\{v_i;i\in\mathbb{N}\}$ is finite, then $\sigma$ is uniformly continuous. In particular, if ${\mathfrak V}$ is finite every $\sigma$ is uniformly continuous.
\end{remark} 

The next result is an extension of \cite[Theorem 4.38]{GH}.

\begin{theorem} Let $(X,{\mathfrak V})$ be a uniformly almost periodic IFS. If $\sigma\in\Sigma$ 
is uniformly continuous, then 
it is equicontinuous.
\end{theorem}
\begin{proof}
	Since $(X,{\mathfrak V})$ is uniformly almost periodic for $\epsilon>0$ given we have that $R(\frac{\epsilon}{3})$ is syndetic. By definition of syndetic, there is $L>0$ such that $$R(\frac{\epsilon}{3})\cap[n,n+L-1]\neq\emptyset,$$
for each $n\in\mathbb{Z}^+$. 
Take $\delta=\delta(\sigma,L,\frac{\epsilon}{3})$ from the definition of uniform continuity. Therefore, for $n\le L$ and $x,y\in\Sigma_\sigma$ with $d(x,y)<\delta$, \[d(v^{\sigma(n)}(x), v^{\sigma(n)}(y))<\epsilon/3.\]
For $n>L$ there is $$j\in R(\frac{\epsilon}{3})\cap[n-L,n-1]$$ and $k\in\{1,...,L\}$ such that $n=j+k$. Hence $d(v^{\sigma(j+k,k)}(x),x)<\frac{\epsilon}{3}$ for each $x\in \Sigma_{\sigma(+\infty,k)}$. Then, if $x,y\in\Sigma_\sigma$ and $d(x,y)<\delta$
\begin{align*}
	d(v^{\sigma(n)}(x),v^{\sigma(n)}(y))&=d(v^{\sigma(j+k)}(x),v^{\sigma(j+k)}(y)) \\ &\leq d(v^{\sigma(j,k)}(v^{\sigma(k)}(x)),v^{\sigma(k)}(x)) + d(v^{\sigma(k)}(x),v^{\sigma(k)}(y)) \\&\quad+ d(v^{\sigma(j,k)}(v^{\sigma(k)}(y)),v^{\sigma(k)}(y)) \\ &<\frac{\epsilon}{3}+\frac{\epsilon}{3}+\frac{\epsilon}{3}=\epsilon.
\end{align*}
Hence 
$\sigma$ is equicontinuous.

\end{proof}

\begin{definition}
We say that $(X,{\mathfrak V})$ has a \emph{uniformly rigid path} if there are $\sigma\in\Sigma$ and a sequence $\{m_k\}$ of natural numbers such that $v^{\sigma(m_k)}\rightarrow Id$ uniformly when $k\rightarrow +\infty$. In this case $\sigma$ is said to be \emph{uniformly rigid}.
\end{definition}

\begin{example} For any homeomorphism $f:X\rightarrow X$ on a compact metric space $X$, with the Gluing Orbit property, we have that the IFS defined by $f$ and the identity on $X$ has the Gluing Orbit property. It is obvious that, if $X$ is not trivial, then the identity is not minimal, but it is equicontinuous and uniformly rigid. If $f$ has zero topological entropy, then the result of \cite{PSun} shows that $f$ is minimal, equicontinuous and uniformly rigid.  
\end{example}


A point $x\in X$ is a \textit{recurrent point of} $(X,{\mathfrak V})$ if for every $\epsilon>0$ and $\sigma\in\Sigma_x$ 
there is $n\in\mathbb N$
such that $$d(v^{\sigma(n)}(x),x)<\epsilon.$$
 The next lemma is based on \cite[Lemma 3.1]{Sun2019}.
\begin{lemma} Let $(X,{\mathfrak V})$ be an iterated function system with the Gluing Orbit property. Assume that there are $x, y\in X$, $\sigma\in\Sigma_x$ and $\delta>0$ such that $$d(v^{\sigma(n)}(x),y)>\delta, \ \forall  n\in\mathbb{Z}^+.$$ Then there is a non-recurrent point of $(X,{\mathfrak V})$.
\end{lemma}

\begin{proof} Let $0<\epsilon<\frac{\delta}{3}$ and $m=M(\epsilon)$ from the gluing orbit property. Define $\mathcal{C}=\{(y,Id, 1),(x,\sigma,+\infty)\}$. By the Gluing Orbit Property, there is $\mathcal{G}=\{t\}$ with $1\le t\le m$, such that $(\mathcal{C},\mathcal{G})$ is $\epsilon$-traced by some $(z,\phi)\in\Psi(X)$. Let $t_0$ be the minimum gap in $[1, m]$ that satisfies this property. So there is $(z, \phi)\in\Psi(X)$ such that $d(z, y)<\epsilon$ and, for each $n\geq 0$,   $$d(v^{\phi(n+t_0)}(z),v^{\sigma(n)}(x))<\epsilon.$$ 
Then $d(v^{\phi(n+t_0)}(z), z)>\delta-2\epsilon$ for all $n\ge 0$. For $1\le s<t_0$, $v^{\phi(s)}(z)\neq z$. Otherwise if $v^{\phi(s)}(z)= z$ for some $1\leq s< t_0$, then for all $n\in\mathbb Z^+$ \[d(v^{\phi(n+t_0, s)}z, v^{\sigma(n)}(x))=d(v^{\phi(n+t_0)}(z),v^{\sigma(n)}(x))<\epsilon,\]
so that with $\mathcal G'=\{t_0-s\}$ and $\phi'=\phi(+\infty, s)$, $(\mathcal C, \mathcal G')$ is $\epsilon$-traced by $(z, \phi')$, then  $t_0-s$ would be a smaller gap than $t_0$.

Therefore, for all $n\in\mathbb N$, with $\lambda=\min\{d(v^{\phi(s)}(z),z); 1\leq s<t_0\}$
\[d(v^{\phi(n)}(z), z)\ge\min\{\delta-2\epsilon, \lambda\}>0.\]

\end{proof}

\begin{lemma} Let $(X,{\mathfrak V})$ be an iterated function system with the Gluing Orbit property. Then there is $p\in X$ such that $\{v^{\sigma(n)}(p);  \sigma\in\Sigma_p, n\in\mathbb{N}\}$ is dense in $X$.
\end{lemma}

\begin{proof} Let $\{U_k\}_{k\in\mathbb{N}}$ be a countable basis of nonempty open sets of X. For each $k\in\mathbb{N}$, define $$A_k=\{x\in X;  \exists \sigma\in\Sigma_x \text{ and }  n\in\mathbb{N} \text{ such that } v^{\sigma(n)}(x)\in U_k\}=\bigcup_{\sigma\in\Sigma}\bigcup_{n\in\mathbb{N}}v^{-\sigma(n)}(U_k). $$ Each $A_k$ is an open set, since it is a union of open sets. By the Gluing Orbit property, each $A_k$ is dense. By Baire Category Theorem the set $T=\bigcap_{k\in\mathbb{N}}A_k$ is nonempty. Fix $p\in T$. We shall prove that $\{v^{\sigma(n)}(p);  \sigma\in\Sigma_p, n\in\mathbb{N}\}$ is dense in $X$. For each $k\in\mathbb{N}$, we have that $p\in A_k$. By definition of $A_k$ there is $\sigma\in\Sigma_p$ and $n\in\mathbb{N}$ such that $v^{\sigma(n)}(p)\in U_k$. Hence, $\{v^{\sigma(n)}(p); \sigma\in\Sigma_p, n\in\mathbb{N}\}$ intersects every basis element $U_k$, and then it is dense in $X$.
	
\end{proof}

Define $Tran(X,{\mathfrak V})$ as the subset of $X$ of points $p$ such that there is $\sigma\in\Sigma_p$ with $\{v^{\sigma(n)}(p); n\in\mathbb{N}\}$ is dense in $X$. In this case we say that $T_p$ is the subset of $\Sigma_p$ such that for each $\sigma\in T_p$ we have that $\{v^{\sigma(n)}(p); n\in\mathbb{N}\}$ is dense in $X$. Note that the last lemma do not implies that $Tran(X,{\mathfrak V})$ is a nonempty set for generalized function systems with the Gluing Orbit Property.

And define $Tran(X,{\mathfrak V},\sigma)\subset Tran(X,{\mathfrak V})$ as the set of all points $p\in X$ such that $\{v^{\sigma(n)}(p); n\in\mathbb{N}\}$ is dense in $X$.  

\begin{lemma}\label{Tec1}
	Suppose that there are $\gamma>0$, $p\in Tran(X,{\mathfrak V},\sigma)$, and $m\in\mathbb{Z}^+$ such that $$d(v^{\sigma(n)}(p),v^{\sigma(n+m)}(p))\leq \gamma$$
	for each $n\in\mathbb{N}$. Then $$d(x,v^{\sigma(m)}(x))\leq\gamma,$$ for every $x\in\Sigma_\sigma$. 
\end{lemma}
\begin{proof} Fix $x\in\Sigma_\sigma$, and take a sequence $\{n_k\}$ such that $$\lim_{k\rightarrow +\infty}v^{\sigma(n_k)}(p)=x.$$
	Hence for each $n\in\mathbb{N}$, we have 
	$$d(x,v^{\sigma(m)}(x))=\lim_{k\rightarrow +\infty}d(v^{\sigma(n_k)}(p),v^{\sigma(m+n_k)}(p))\leq\gamma.$$ 
\end{proof}


\begin{lemma}
	Suppose that $\sigma\in\Sigma$ is not uniformly rigid and $Tran(X,{\mathfrak V},\sigma)\neq\emptyset$. 
	Then there is $\gamma>0$ such that for each $p\in Tran(X,{\mathfrak V},\sigma)$, and  $m\in\mathbb N$, there is $\tau=\tau(p,\sigma,m)\in\mathbb{N}$ such that   $$d(v^{\sigma(\tau)}(p),v^{\sigma(\tau+m)}(p))>\gamma.$$ 
\end{lemma}
\begin{proof} Assume that for each $k\in\mathbb{Z}$, there are $p_k\in Tran(X,{\mathfrak V},\sigma)$ and $m_k\in\mathbb N$ such that $$d(v^{\sigma(n)}(p_k),v^{\sigma(n+m_k)}(p_k))\leq \frac{1}{2^k}, $$
for every $n\in\mathbb{N}$. By Lemma \ref{Tec1}, for each $x\in\Sigma_\sigma$, we have
$$d(x,v^{\sigma(m_k)}(x))\leq \frac{1}{2^k}, $$  
for every $k\in\mathbb{N}$. Therefore, $v^{\sigma(m_k)}\rightarrow Id$ uniformly. Hence $\sigma$ is uniformly rigid. 
\end{proof}

\section{Proof of Theorem \ref{T3}}\label{sec:t3}

Now we prove Theorem \ref{T3}, namely, that if $(X,{\mathfrak V})$ is an iterated function system with the Gluing Orbit property, 
with $Tran(X,{\mathfrak V},\sigma)\neq\emptyset$ for some $\sigma\in\Sigma$ that is not uniformly rigid, then $h(X,{\mathfrak V})>0$.  The proof is based on \cite[Proposition 3.3]{PSun}.

\begin{proof} 
Let $\sigma$ be as in the condition and let  $p\in Tran(X,{\mathfrak V},\sigma)$. By the previous lemma, there is $\gamma>0$ such that for every $k\in\mathbb N$, there is $\tau_k$ with 
	\begin{equation}\label{eq:2} d(v^{\sigma(\tau_k)}(p),v^{\sigma(\tau_k+k)}(p))>\gamma. 
	\end{equation}	
For any $\varepsilon\in (0,\frac{\gamma}{3})$, let $M$ be as in the definition of Gluing Orbit property and  let $$T=2M+\max\{\tau_k; k=1,...,2M\}.$$ Define, $m_1=T+M$ and $m_2=T$. And for each sequence $a=\{a(k)\}_{k=1}^{+\infty}\in\{1, 2\}^{\mathbb N}$, 
 denote
$$\mathcal{C}_a=\{(p,\sigma,m_{a(k)}+1)\}_{k=1}^{+\infty}.$$  
By the Gluing orbit property there are $z_a\in X$, $\sigma_a\in\Sigma_{z_a}$ and $\mathcal{G}_a=\{t_k(a)\}_{k=1}^{+\infty}\in\Sigma_M$ such that $(\mathcal{C}_a,\mathcal{G}_a)$ is $\varepsilon$-traced by $(z_a,\sigma_a)$.

\textbf{Claim:} Fix $N\in\mathbb{N}$. If $a(n)\neq b(n)$ for some $1\leq n\leq N$, then $(z_a,\sigma_a)$ and $(z_b,\sigma_b)$ are $((N+1)(T+2M),\varepsilon)$-separated.

If the claim is correct, for each $N$ there is an $((N+1)(T+2M),\varepsilon)$-separated set with cardinality $2^N$. Hence
$$h(X,{\mathfrak V})\geq\lim_{N\rightarrow\infty}\frac{\ln2^N}{(N+1)(T+2M)}=\frac{\ln2}{T+2M}>0.$$
Now we shall prove the claim. Without loss of generality, we can assume that $a(k)=b(k)$ for each $k\in\{1, 2, ..., n-1\}$, and $a(n)=1\neq 2=b(n)$.

\textbf{Case 1:} The "gaps" are the same, i.e., $$t_k(a)=t_k(b) \text{ for each } k\in\{1,...,n-1\}.$$

Therefore,
$$s_n=\sum_{k=1}^{n-1}(m_{a(k)}+t_k(a))=\sum_{k=1}^{n-1}(m_{b(k)}+t_k(b))\leq (n-1)(T+2M).$$
Hence, the orbits of $(z_a,\sigma_a)$ and $(z_b,\sigma_b)$ start to trace the $n$-th orbit segments of $\mathcal{C}_a$ and $\mathcal{C}_b$, respectively, at the same time. Since the $n$-th orbit segments of $\mathcal{C}_a$ and $\mathcal{C}_b$ have different lengths, and $(n+1)$-th orbit segments has lengths no less than $T+1$, we have
\begin{equation}\label{eq:3}
	d(v^{\sigma_a(s_n+T+M+t_n(a)+l)}(z_a),v^{\sigma(l)}(p))\leq\varepsilon,  
\end{equation} 
\begin{equation}\label{eq:4}
	d(v^{\sigma_b(s_n+T+t_n(b)+l)}(z_b),v^{\sigma(l)}(p))\leq\varepsilon,   
\end{equation}
for each $l\in\{1,...,T\}$.    

Let $r=M+(t_n(a)-t_n(b))$ and then we have that $r\in [1,2M]$. By \eqref{eq:2}  
we have that $$d(v^{\sigma(\tau_r)}(p),v^{\sigma(\tau_r+r)}(p))>\gamma.$$
And, by definition of $T$, we have,
$$ r+\tau_r<2M+\tau_r\leq T.$$
Therefore, we can apply \eqref{eq:3} and \eqref{eq:4} for $\tau_r$ and $r+\tau_r$ respectively, and hence
\begin{align*} 
	&\quad d(v^{\sigma_a(s_n+T+M+t_n(a)+\tau_r)}(z_a),v^{\sigma_b(s_n+T+M+t_n(a)+\tau_r)}(z_b)) \\
	&\geq d(v^{\sigma(\tau_r)}(p),v^{\sigma(\tau_r+r)}(p))-d(v^{\sigma_a(s_n+T+M+t_n(a)+\tau_r)}(z_a),v^{\sigma(\tau_r)}(p))  \\&\quad-d(v^{\sigma_b(s_n+T+M+t_n(a)+\tau_r)}(z_b),v^{\sigma(\tau_r+r)}(p))\\&>\gamma-2\varepsilon>\varepsilon.
\end{align*}
Moreover,
\begin{multline*}
s_n+T+M+t_n(a)+\tau_r\leq (n-1)(T+2M)+T+2M+(T-2M)\\<(N+1)(T+2M), \end{multline*}
that is an estimate of the separation time.

\textbf{Case 2:} If $t_k(a)\neq t_k(b)$ for some $k\in\{1,...,n-1\}$ then define $$L=\min \{k\in\mathbb{N};t_k(a)\neq t_k(b)\}.$$
Note that $L\leq n-1$. Without loss of generality we may assume that $t_L(a)>t_L(b)$ and let $r=t_L(a)-t_L(b)\in \{1,...,M-1\}$. Hence, by \eqref{eq:2}, we have $$d(v^{\sigma(\tau_r)}(p),v^{\sigma(\tau_r+r)}(p))>\gamma.$$ 
Hence, $$s=\sum_{k=1}^L(m_{a(k)}+t_k(a))=\sum_{k=1}^L(m_{b(k)}+t_k(b))+r\leq L(T+2M).$$
Therefore, for each $l\in\{0,1,...,T\}$ we have that
\begin{equation*}
	d(v^{\sigma_a(s+l)}(z_a),v^{\sigma(l)}(p))\leq \varepsilon \text{ and } 
	d(v^{\sigma_b(s-r+l)}(z_b),v^{\sigma(l)}(p))\leq \varepsilon.	
\end{equation*}
Then, 
\begin{align*}
	&\quad d(v^{\sigma_a(s+\tau_r)}(z_a),v^{\sigma_b(s+\tau_r)}(z_b))\\
	&\geq d(v^{\sigma(\tau_r)}(p),v^{\sigma(r+\tau_r)}(p)) 
	-d(v^{\sigma_a(s+\tau_r)}(z_a),v^{\sigma(\tau_r)}(p))\\ &\quad -d(v^{\sigma_b(s+\tau_r)}(z_b),v^{\sigma(r+\tau_r)}(p))\\
	&>\gamma-2\varepsilon>\varepsilon.
\end{align*}
This concludes the proof.

\end{proof}

\noindent {\bf Acknowledgments.} W.~Cordeiro thanks the Federal University of Rio de Janeiro for its hospitality during the preparation of this work. %

%
%
%
%
%
%
%
%
%
%
%
%
%
%
%
%
%


\begin{thebibliography}{1}

\bibitem{CL} Dandan Cheng and Zhiming Li.
\newblock Upper metric mean dimensions for impulsive semi-flows.
\newblock {\em J. Differential Equations}, 311:81--97, 2022. 

\bibitem{CordeiroDenkerYuri2015}
Welington Cordeiro, Manfred Denker, and Michiko Yuri.
\newblock A note on specification for iterated function systems.
\newblock {\em Discrete Contin. Dyn. Syst. Ser. B}, 20(10):3475--3485, 2015.

\bibitem{Denker2023}
Manfred Denker.
\newblock General iterated function systems: {H}ausdorff dimension.
\newblock {\em Chaos}, 33(3):Paper No. 033118, 5, 2023.

\bibitem{DenkerYuri2015}
Manfred Denker and Michiko Yuri.
\newblock Conformal families of measures for general iterated function systems.
\newblock In {\em Recent trends in ergodic theory and dynamical systems},
  volume 631 of {\em Contemp. Math.}, pages 93--108. Amer. Math. Soc.,
  Providence, RI, 2015.

\bibitem{GH}
Walter Helbig Gottschalk and Gustav Arnold Hedlund.
\newblock {\em Topological dynamics}.
\newblock volume 36 of {\em Amer. Math. Soc. Colloq. Publ.}, Amer. Math. Soc., Providence, RI, 1955.

\bibitem{G} 
Misha Gromov. 
\newblock{Topological invariants of dynamical systems and spaces of holomorphic maps: I}. 
\newblock {\em Math. Phys. Anal. Geom.}, 2:323--415, 1999. 

\bibitem{IvandovIlliadisTuzhilin2016} 
Alexander Ivanov, Stavros Iliadis, and Alexey Tuzhilin.
\newblock Realizations of Gromov–Hausdorff distance.
\newblock {\em arXiv}:1603.08850, 2016.

\bibitem{LindenstraussWeiss2000}
Elon Lindenstrauss and Benjamin Weiss.
\newblock Mean topological dimension.
\newblock {\em Israel J. Math.}, 115:1--24, 2000.

\bibitem{RA}
Fagner B. Rodrigues and Jeovanny Muentes Acevedo. \newblock Mean dimension and metric mean dimension for non-autonomous dynamical systems. 
\newblock {\em J. Dyn. Control Syst.}, 28:697--723, 2022. 

\bibitem{Sun2019}
Peng Sun.
\newblock Minimality and gluing orbit property.
\newblock {\em Discrete Contin. Dyn. Syst.}, 39(7):4041--4056, 2019.

\bibitem{PSun}
Peng Sun.
\newblock Zero-entropy dynamical systems with the gluing orbit property.
\newblock {\em Adv. Math.}, 372:107294, 2020.

\bibitem{TYM} 
Yanjie Tang, Xiaojiang Ye and Dongkui Ma.
\newblock Metric mean dimension of free semigroup actions for non-compact sets. 
\newblock {\em J. Dyn. Control. Syst.}, 30:12, 2024.

\bibitem{Walters1982}
Peter Walters.
\newblock {\em An introduction to ergodic theory}, volume~79 of {\em Graduate Texts in Mathematics}.
\newblock Springer-Verlag, New York-Berlin, 1982.

\end{thebibliography}
\end{document}